\documentclass[11pt,reqno]{amsart}
\usepackage{microtype}

\usepackage[dvipsnames]{xcolor}
\usepackage
[colorlinks=true,linkcolor=Maroon,citecolor=OliveGreen
]
{hyperref}
\usepackage{amsmath, amsfonts,amsthm,amssymb,verbatim,
    color,multirow,booktabs,mathdots,bm}
\usepackage[shortlabels]{enumitem}
\setlist[enumerate]{label={(\arabic*)}}
\usepackage[capitalize]{cleveref}
\crefname{equation}{}{}

\usepackage[abbrev, alphabetic]{amsrefs}  

\newtheorem{theorem}{Theorem}
\newtheorem{lemma}[theorem]{Lemma}

\newtheorem*{theorem*}{Theorem}

\theoremstyle{definition}

\theoremstyle{remark}
\newtheorem{remark}[theorem]{Remark}

\def\E{\opr{E}} 

\newcommand\opr[1]{\operatorname{#1}}

\def\Aut{\opr{Aut}}

\newcommand\br[1]{{\left(#1\right)}}

\newcommand\floor[1]{\left\lfloor#1\right\rfloor}

\usepackage{todonotes}

\def\Syl{\opr{Syl}}
\def\CC{\mathcal{C}}
\def\supp{\opr{supp}}
\def\sm{\smallsetminus}

\author{Sean Eberhard}

\address{\parbox{\linewidth}{Mathematics Institute, Zeeman Building, University of Warwick, UK \vspace{0.1cm}}}
\email{sean.eberhard@warwick.ac.uk}

\thanks{SE is supported by the Royal Society.}

\begin{document}
\title{Intersections of Sylow $2$-subgroups in symmetric groups}

\begin{abstract}
    We compute the asymptotic probability that a random pair of Sylow $2$-subgroups in $S_n$ or $A_n$ intersects trivially.
    This calculation complements recent work of Diaconis, Giannelli, Guralnick, Law, Navarro, Sambale, and Spink.
\end{abstract}

\maketitle


Lately there has been a surge of interest in intersections of Sylow subgroups in finite groups.
A particularly bold general conjecture was proposed recently by Lisi and Sabatini~\cite{LS}.
In the special case of symmetric groups, the main theorem of \cite{7AP} is the striking result that if $p > 2$ then pairs of random Sylow $p$-subgroups tend to intersect trivially, while if $p = 2$ then the intersection is nontrivial with probability \emph{at least} approximately $1 - e^{-1/2} + o(1)$.
One of the key ideas in the proof is to use a theorem of the author and Daniele Garzoni on random generation of $A_n$ by elements of fixed cycle type~\cite{EG}.
In this note we prove the following result about intersection of Sylow $2$-subgroups, complementing the theorem of \cite{7AP}.

\begin{theorem}
    \label{thm}
    Let $P \in \Syl_2(S_n)$ and let $x \in S_n$ be chosen uniformly at random. Then
    \[
        \Pr\br{P \cap P^x = 1} = e^{-1/2} + O(1/n).
    \]
    If $P \in \Syl_2(A_n)$ then similarly
    \[
        \Pr\br{P \cap P^x = 1} = (3/2) e^{-1/2} + O(1/n).
    \]
\end{theorem}

Let $P \in \Syl_2(S_n)$.
The structure of $P$ is summarized in \cite{7AP}*{\S 2.1}.
Write $n = 2^{e_1} + \cdots + 2^{e_m}$ where $e_1 > \cdots > e_m \ge 0$.
Then $P$ can be identified with the automorphism group of a forest $F$ of $m$ complete binary trees $T_1, \dots, T_m$ of heights $e_1, \dots, e_m$ whose leaves are labelled by $\{1, \dots, n\}$:
\[
    P \cong \Aut(F) \cong \Aut(T_1) \times \cdots \times \Aut(T_m).
\]
Each factor $\Aut(T_i)$ is isomorphic to the iterated wreath product $C_2 \wr \cdots \wr C_2$ of $e_i$ copies of $C_2$.
For any vertex $v$ of $F$, we denote by $T_v$ denote the complete subtree rooted at $v$, and we define the \emph{height} of $v$ to be the height of $T_v$.
The group $P$ acts transitively on the vertices of $T_i$ of any given height.
Let $A \le P$ be the elementary abelian subgroup of rank $\floor{n/2}$ that stabilizes the vertices of positive height.

If $g \in S_n$ we denote by $\supp(g)$ the size of its support.

\begin{lemma}
    \label{key-lemma}
    Let $s \ge 0$ be an even integer.
    \begin{enumerate}
        \item $\#\{g \in A : \supp(g) = s\} \le O(1)^s (n/s)^{s/2}$.
        \item $\#\{g \in P \sm A : \supp(g) = s\} \le O(1)^s (n/s)^{s/2 - 1}$.
    \end{enumerate}
\end{lemma}
\begin{proof}
    \emph{(1)}
    Since $A$ is generated by $\floor{n/2}$ disjoint transpositions, the number of $g \in A$ of support $s$ is exactly $\binom{\floor{n/2}}{s/2}$,
    which is at most $(en/s)^{s/2}$.
    
    \emph{(2)}
    To each $g \in P$ we can associate the subforest $F_g$ of $F$
    induced by the vertices of $F$ that are either (a) moved by $g$ or (b) not moved by $g$ but adjacent to a vertex moved by $g$.
    Let $v_1, \dots, v_k$ be the vertices of type (b), and let $h_i > 0$ be the height of $v_i$.
    Then $F_g$ is the disjoint union of the subtrees $T_{v_1}, \dots, T_{v_k}$.
    The support of $g$ is $s = 2^{h_1} + \cdots + 2^{h_k}$, and $g \in A$ if and only if $h_i = 1$ for all $i$.
    
    Consider the ways of choosing $h_1, \dots, h_k$ and $v_1, \dots, v_k$ and $g$.
    First, assume $h_1, \dots, h_k$ are given, where $h_1 \ge \cdots \ge h_k > 0$.
    For each $h > 0$ let $m_h$ be the number of indices $i$ with $h_i = h$.
    The number of vertices of $F$ of height $h$ is at most $n / 2^h \le n$, so the number of choices for $v_1, \dots, v_k$ is at most
    \[
        \frac{\prod_{i=1}^k (n / 2^{h_i})}{\prod_{h>0} m_h!} \le \prod_{h > 0} \frac{n^{m_h}}{m_h!}.
    \]
    Next, the restriction of $g$ to $T_{v_i}$ must be an automorphism swapping the two main branches, of which there are $2^{2^{h_i} - 2}$, so the number of choices for $g$ given $v_1, \dots, v_k$ is $\prod_{i=1}^k 2^{2^{h_i} - 2} \le 2^s$.
    Thus the number of possibilities for $g$ given $h_1, \dots, h_k$, or equivalently given $(m_h)$, is bounded by
    \[
        2^s \prod_{h>0} \frac{n^{m_h}}{m_h!}.
    \]
    It follows that the number of $g \in P \sm A$ of support $s$ is bounded by
    \begin{equation}
        \label{eq:big-sum}
        2^s \sum_{(m_h)}
        \prod_{h>0} \frac{n^{m_h}}{m_h!},
    \end{equation}
    where the sum goes over all nonnegative integer solutions $(m_h)_{h>0}$ to
    \begin{equation}
        \label{eq:m_h}
        \sum_{h>0} 2^{h-1} m_h = s/2,
        \hspace{4em}
        m_1 < s/2.
    \end{equation}
    
    Let $(m_h)$ be a solution to \eqref{eq:m_h}.
    Write $\sum m_h = s/2 - t$. Then $t = \sum_{h>1} (2^{h-1} - 1) m_h$.
    In particular the number of solutions $(m_h)$ with a given $t$ is bounded by the number of partitions of $t$, which is at most $2^t$. Also,
    \[
        s/2 = \sum_{h>0} 2^{h-1} m_h \le m_1 + 2 \sum_{h>1} (2^{h-1} - 1) m_h = m_1 + 2t,
    \]
    so $m_1 \ge s/2 - 2t$. Therefore our sum \eqref{eq:big-sum} is bounded by
    \[
        2^s \sum_{0 < t \le s/2} \frac{2^t n^{s/2-t}}{\max(s/2 - 2t, 0)!}.
    \]
    If $2 n^{-1} (s/2 - 2) (s/2 - 3) < 1$ then the largest term in the sum is the $t=1$ term, so the sum is bounded by
    \[
        s 2^{s+1} \frac{n^{s/2-1}}{(s/2 - 2)!} = O(1)^s (n/s)^{s/2-1}.
    \]
    This completes the proof for say $s < (2n)^{1/2}$.
    
    For $s \ge (2n)^{1/2}$ there is a phase transition and we analyze differently. Ignoring the restriction $m_1 < s/2$, the sum \eqref{eq:big-sum} is the coefficient of $X^{s/2}$ in the generating function
    \[
        F(X)
        = 2^s \prod_{h>0} \br{\sum_{m_h\ge 0} \frac{n^{m_h}}{m_h!} X^{2^{h-1} m_h}}
        = 2^s \exp\br{ \sum_{h>0} n X^{2^{h-1}}}.
    \]
    Since this power series has nonnegative coefficients, for any real $x > 0$ the coefficient of $X^{s/2}$ is bounded by
    \[
        x^{-s/2} F(x) = x^{-s/2} 2^s \exp\br{\sum_{h>0} n x^{2^{h-1}}}.
    \]
    Set $x = s/(2n)$. Then the argument of the exponential is
    \[
        \sum_{h>0} n (s/2n)^{2^{h-1}} = \frac{s}{2} \sum_{h>0} (s/2n)^{2^{h-1} - 1} \le \frac{s}{2} \sum_{e \ge 0} 2^{-e} = s.
    \]
    Therefore the sum \eqref{eq:big-sum} is bounded by $O(1)^s (n/s)^{s/2}$, which is the same as $O(1)^s (n/s)^{s/2-1}$ in this region, so this completes the proof.
\end{proof}

The \emph{support} of a partition $\lambda \vdash n$ is the sum of its nonunit parts. It is clear that there are at most $2^s$ partitions of $n$ of support $s$.

\begin{lemma}
    Let $x \in S_n$ be uniformly random. Then $P \cap P^x = A \cap A^x$ with probability $1 - O(1/n)$.
\end{lemma}
\begin{proof}
    It suffices to prove that $P \cap P^x \subseteq A$ with probability $1 - O(1/n)$, as $P \cap P^x \subseteq A^x$ with the same probability.
    The expected size of $P \cap P^x \sm A$ is
    \[
        \E|P \cap P^x \sm A|
        = \sum_{\lambda \vdash n} \frac{|\CC_\lambda \cap P| \ |\CC_\lambda \cap P \sm A|}{|\CC_\lambda|},
    \]
    where $\CC_\lambda \subset S_n$ is the class of elements of type $\lambda$,
    and obviously we may restrict to those $\lambda$ whose parts are powers of $2$, and in particular even support.
    If $\lambda$ has support $s$ then $|\CC_\lambda| \ge O(1)^{-s} s^{-s/2} n^s$.
    Indeed, if $g \in \CC_\lambda \cap S_s$ then $|g^{S_n}| = \binom{n}{s} |g^{S_s}| \ge (n/s)^s |g^{S_s}|$ and $|g^{S_s}| \ge O(1)^{-s} s^{s/2}$ (see for example \cite{EG}*{Lemma~4.7}).
    Therefore from \Cref{key-lemma} we obtain the bound
    \[
        \E|P \cap P^x \sm A|
        = \sum_{s=2}^n O(1)^s \frac{(n/s)^{s-1}}{s^{-s/2} n^s}
        = \frac1n \sum_{s=2}^\infty O(1)^s s^{-s/2}
        = O(1/n).
    \]
    It follows that $P \cap P^x \sm A$ is nonempty with probability $O(1/n)$.
\end{proof}

\begin{proof}[Proof of \Cref{thm}]
    By the previous lemma we have $P \cap P^x = A \cap A^x$ with probability $1 - O(1/n)$, so it suffices to understand $A \cap A^x$. By \cite{7AP}*{\S 4.2}, $A \cap A^x$ is elementary abelian of rank $W$ where $W$ is a random variable that is asymptotically Poisson with parameter $1/2$.
    In fact it is routine to verify that the convergence has rate $O(1/n)$.
    Therefore
    \[
        \Pr(P \cap P^x = 1) = \Pr(W = 0) + O(1/n) = e^{-1/2} + O(1/n).
    \]
    Similarly,
    \[
        \Pr(P \cap P^x \cap A_n = 1) = \Pr(W \le 1) + O(1/n) = (3/2) e^{-1/2} + O(1/n).
    \]
    This completes the proof.
\end{proof}


\begin{remark}
    For comparison, if $G$ is a large simple group of Lie type and $P \in \Syl_p(G)$, it is most likely true that $\Pr(P \cap P^x = 1) = 1 - o(1)$ unless $p$ is the defining characteristic of $G$ and $p$ is small compared to the rank of $G$.
    If $G = \opr{PSL}_n(p)$ and $p$ is bounded then $\Pr(P \cap P^x = 1)$ tends to zero, but comparatively slowly.
    These observations suggest a probabilistic approach to the problem of Vdovin~\cite{kourovka}*{15.40}.
\end{remark}

\bibliography{refs}

\end{document}